\documentclass[11pt,article,reqno]{amsart}
\usepackage{amsmath, amsthm, amsfonts,graphicx}
\usepackage{subfigure}
\usepackage{amssymb,url}
\usepackage{adjustbox}
\usepackage{hyperref}
\usepackage[top=29mm,bottom=29mm,left=30mm,right=30mm]{geometry}
\usepackage{mathrsfs}

\usepackage{geometry}
\usepackage{color}
\usepackage{mathtools}
\makeatletter

\def\section{\@startsection{section}{1}%
\z@{1\linespacing\@plus\linespacing}{1\linespacing}%
{\bf\centering}}
\def\subsection{\@startsection{subsection}{0}%
\z@{\linespacing\@plus\linespacing}{\linespacing}%
{\bf}}
\def\subsubsection{\@startsection{subsubsection}{0}%
\z@{\linespacing\@plus\linespacing}{\linespacing}%
{\bf}}
\makeatother

\usepackage{mathrsfs}

\usepackage[square, numbers, sort]{natbib}
\usepackage{bm}
\usepackage{ragged2e}
\usepackage{bm}
\usepackage{multirow, booktabs}
\usepackage{amssymb,epsfig,epstopdf,url,array}

\theoremstyle{plain}
\newtheorem{thm}{Theorem}[section]
\newtheorem{lem}[thm]{Lemma}
\newtheorem{prop}[thm]{Proposition}

\theoremstyle{remark}
\newtheorem{definition}[thm]{Definition}
\newtheorem{rem}{Remark}

\numberwithin{equation}{section}

\begin{document}
\title[A finite-state stationary process with long-range dependence]{A finite-state stationary process with long-range dependence and fractional multinomial distribution}

\author{}
\address{Jeonghwa Lee, Department of Statistics, Truman State University, USA}
\email{jlee@truman.edu}

\author{Jeonghwa Lee}

\thanks{
Jeonghwa Lee: {Department of Statistics, Truman State University, USA,}
{jlee@truman.edu}}
\begin{abstract}
 We propose a discrete-time, finite-state stationary process that can possess long-range dependence. Among the  interesting features of this process is that each state can have different long-term dependency, i.e., the  indicator sequence can have different Hurst index for different states. Also, inter-arrival time for each state follows heavy tail distribution, with different states showing different tail behavior. A possible application of this process is to model over-dispersed multinomial distribution. In particular, we define fractional multinomial distribution from our model.
\vspace{5pt}
\end{abstract}

\maketitle

\baselineskip 0.5 cm

\bigskip \medskip

\section{Introduction}

Long-range dependence (LRD) refers to a phenomenon where correlation decays slowly with the time lag in  stationary process in a way that correlation function  is no longer summable.  This phenomenon was first observed by Hurst \cite{Hu1,Hu2} and since then it has been observed in many fields such as economics, hydrology, internet traffic, queueing networks, etc.(\cite{BE, DEL, MAJ, Sam})  
In a second order stationary process, LRD can be measured by Hurst index $H$ \cite{Dal1, Dal2},
\[H=\inf \{h: \limsup_{n\rightarrow \infty} n^{-2h+1} \sum_{k=1}^{n} cov(X_1,X_k) <\infty\}. \]
Note that $H\in (0,1),$ and  
if $H\in(1/2,1),$ the process possesses long-memory property.

Among the well known stochastic processes that are stationary  and possess long-range dependence are
fractional Gaussian noise (FGN) \cite{Man} and fractional autoregressive integrated moving average processes (FARIMA) \cite{HOS,HOS2}. 

\par  Fractional Gaussian noise $X_j$ is a mean-zero, stationary Gaussian process with covariance function \[\gamma(j):=cov(X_0,X_j)=\frac{var(X_0)}{2}(|j+1|^{2H}-2|j|^{2H}+|j-1|^{2H})\] 
where $H\in(0,1)$ is Hurst parameter.
The covariance function  obeys power law with exponent $2H-2$ for large lag, \begin{equation*}
     \gamma(j)\sim var(X_0)H(2H-1)j^{2H-2} \text{ as } j\rightarrow \infty.   \end{equation*}
If $H\in (1/2, 1),$ then the
covariance function decreases  slowly with power law, and
$\sum_j \gamma(j)=\infty,$ i.e., it has long-memory property. 

A FARIMA(p,d,q) process $\{X_t\}$ is the solution of \[\phi(B) \triangledown^d X_t=\theta(B)\epsilon_t,  \] where $p,q$ are positive integers, $d$ is real, $B$ is the backward shift, $BX_t=X_{t-1},$ and fractional- differencing operator  $\triangledown^d$, autoregressive operator $\phi$, and moving average operator $\theta$ are, respectively,
\begin{align*}
    \triangledown^d&=(1-B)^d=\sum_{k=1}^{\infty}\frac{d(d-1)\cdots(d+1-k)}{k!}(-B)^k,\\
    \phi(B)&=1-\phi_1B-\phi_2 B^2\cdots -\phi_p B^p,  \\
    \theta(B)&=1-\theta_1 B-\theta_2 B^2\cdots -\theta_q B^q .
    \end{align*}
$\{\epsilon_t\}$ is white-noise process which consists of iid random variables with the finite second moment. Here, the parameter $d$ manages long-term dependence structure, and by its relation to Hurst index, $H=d+1/2,$ $d\in(0,1/2)$ corresponds to the long-range dependence in FARIMA process.

Another class of stationary process that can possess long-range dependence is from countable-state Markov process  \cite{CAR}. In a stationary, positive recurrent, irreducible, aperiodic Markov chain, the indicator sequence of visits to a certain state is long-range dependence if and only if return time to the state has infinite second moment,  and this is possible only when the Markov chain has infinite state space. Moreover, if one state has  the infinite second moment of return time, then all the other states also have the infinite second moment of return time, and all the states have the same rate of dependency, that is, the indicator sequence of each  state  is long-range dependence with the same Hurst index.

In this paper, we develop a discrete-time finite-state stationary process that can possess long-range dependence. 
We define a stationary process $\{X_{i}, i\in \mathbb{N}\}$ where the number of possible outcomes of  $X_i$ is finite, $S=\{0,1,\cdots,m\}$ for any $ m\in \mathbb{N},$    and  for $k=1,2,\cdots ,m,$   \begin{equation}
cov(I_{\{X_i=k\}}, I_{\{X_j=k\}})=c_k'  |i-j|^{2H_k-2},  \end{equation} for any $i,j\in \mathbb{N}, i\neq j,$  and some constants $c_k' \in \mathbb{R}_{+}, H_k\in (0,1).$ This leads to
\begin{equation} cov(X_i, X_j)\sim c'_{k'}|i-j|^{2H_{k'}-2} \hspace{10pt}\text{as $|i-j|\rightarrow \infty$,} \end{equation} where $k'=argmax_k\{H_{k};k=1,\cdots,m\}.$
If $H_{k'}=\max\{H_k;k=1,\cdots,m\}\in (1/2,1)$, (1.2) implies that as $n\rightarrow \infty, $ $\sum_{i=1}^{n} cov(X_1,X_i)$ diverges with the rate  of $ |n|^{2H_{k'}-1}$, and the process is said to have long-memory with Hurst parameter $H_{k'}$. Also, from (1.1), for $k=\{1,\cdots,m\},$ the process $\{I_{\{X_i=k\}}; i=1,2,\cdots\}$ is long-range dependence if $H_k \in (1/2,1).$ In particular, if $H_i\neq H_j,$ then the states $``i"$ and $``j"$ produce different level of dependence. For example, if $H_i<1/2<H_j,$ then the state $``j"$ produces long-memory counting process whereas state $``i"$ produces short-memory process. 

A possible application of our stochastic process is to model over-dispersed multinomial distribution. In multinomial distribution, there are $n$ trials, each trial results in one of finite outcomes, and the outcomes of trials are independent and identically distributed. When applying  multinomial model to a real data, it is often observed that the variance is larger than what is assumed to be, which is called over-dispersion, due to the violation of  the assumption that trials are independent and have identical distribution  \cite{Dea, Poo}, and there have been several ways to model overdispersed  multinomial distribution \cite{Afr,Afr1,Lan,Mos}.
\par Our stochastic process provides a new method to model over-dispersed multinomial distribution  by introducing  dependency among trials. In particular, the variance of the number of a certain outcome among $n$ trials is asymptotically proportional to fractional exponent of $n,$ from which we define \[ Y_k:=\sum_{i=1}^n I_{\{X_i=k\}} 
\text{ for } k=1,2,\cdots,m,\] and call the distribution of $(Y_1,Y_2,\cdots,Y_m )$ fractional multinomial distribution.

\par The work in this paper is an extension of the earlier work of generalized Bernoulli process \cite{LEE}, and the process in this paper is reduced to the generalized Bernoulli process if there are only two states in the possible outcomes of $X_i$, e.g., $S=\{0,1\}$. 
\par In Section 2, finite state stationary process that can possess long-range dependence is developed. In Section 3, the properties of our model are investigated with regard to tail behavior and moments of inter-arrival time of a certain state $``k"$, and conditional probability of observing a state $``k"$ given the past observations in the process. In Section 4, fractional multinomial distribution is defined, followed by conclusion in Section 5. Some proofs of proposition and theorems are in Section 6.
 
\par Throughout this paper, $\{i, i_0, i_1,\cdots\},\{i',i_0',i_1',\cdots \} \subset \mathbb{N},$ with $i_0<i_1<i_2<\cdots,$ and $i_0'<i_1'<i_2'<\cdots.$ For any set $A=\{i_0, i_1,\cdots,i_n \},$ $|A|=n+1,$ the number of elements in the set $A,$ and for the empty set, we define $|\emptyset|=0.$ 

\section{Finite-state stationary process with long-range dependence}

We will define stationary process $\{X_{i}, i\in \mathbb{N}\}$ where the set of possible outcomes of  $X_i$ is finite, $S=\{0,1,\cdots,m\},$ for $m\in \mathbb{N}$, with the probability that we observe a state $``k"$ at time $i$ is $P(X_i=k)=p_k>0,$ for $k=0,1,\cdots,m,$ and $\sum_{k=0}^{m}p_k=1.$ 

For any set $A=\{i_0,i_1,\cdots,i_n\} \subset \mathbb{N}$, define the operator \begin{align*}
    L_{H,p,c}^*(A):= p\prod_{j=1,\cdots,n}(p+c|i_j-i_{j-1}|^{2H-2}) .
\end{align*} 
If $A=\emptyset,$ define $L_{H,p,c}^*(A):=1,$ and if $A=\{i_0\},  L_{H,p,c}^*(A):=p.$

Let ${\bf H}= (H_1,H_2,\cdots,H_m), {\bf p}=(p_1,p_2,\cdots, p_m), {\bf c}=(c_1,c_2,\cdots, c_m) $ be vectors of length $m,$ and $ {\bf H},{\bf p},{\bf c}\in (0,1)^m.$ We are now ready to define the following operators.
\begin{definition}

Let $A_0, A_1, \cdots, A_m \subset \mathbb{N}$ be pairwise disjoint, and $A_0=n'>0.$ Define
\begin{align*}
L_{\bf H, \bf p, \bf c}^*(A_1,A_2,\cdots ,A_m):=\prod_{k=1,\cdots m } L_{H_k,p_k, c_k}^*(A_k),\end{align*} and
\begin{align*}
D_{\bf H, \bf p, \bf c}^*(A_1,A_2,\cdots ,A_m;A_0):=\sum_{\ell=0}^{n'} (-1)^{\ell}\sum_{\substack{|B|=\ell \\ B\subset A_0}}\sum_{\substack{B_i \subset B \\ B_i\cap B_j=\emptyset\\ \cup B_i=B}}L_{\bf H, \bf p, \bf c}^*(A_1\cup B_1,A_2 \cup B_2,\cdots ,A_m \cup B_m).
\end{align*}
\end{definition}
For ease of notation, we denote $D^*_{\bf H, \bf p, \bf c}, $ $L^*_{\bf H, \bf p, \bf c}, $ and 
$L^*_{H_k,  p_k,  c_k}$ by ${\bf D}^*,{\bf L}^*, { L^*_k}, $ respectively.
Note that if $A_0=\{i_0\},$ \begin{equation}
    {\bf D}^*(A_1,A_2,\cdots,A_m;A_0)=\prod_{k=1,\cdots, m } L_{k}^*(A_k)\bigg(1-\sum_{k'=1}^{m} \frac{L_{k'}^*(A_{k'}\cup\{i_0\})}{L_{k'}^*(A_{k'})}   \bigg).\end{equation}
For any pairwise disjoint sets $A_0, A_1,\cdots A_m \subset \mathbb{N},$ if
${\bf D}^*(A_1,A_2,\cdots,A_m;A_0)
>0,$ then $\{X_i; i\in\mathbb{N}\}$ is well defined stationary process with the following  probabilities. \begin{align}
    &P( \cap_{i\in A_k} \{ X_i=k\})= L^*_k(A_k) , \text{ for } k=1,\cdots,m,\\ &
P(\cap_{k=1,\cdots,m} \cap_{i\in A_k} \{ X_i=k\})=\prod_{k=1,\cdots, m } L^*_k(A_k),\\&
P( \cap_{k=0,\cdots,m} \cap_{i\in A_k} \{ X_i=k\})={\bf D}^*(A_1,A_2,\cdots,A_m;A_0).\end{align}

In particular, if the stationary process with the probability above is well defined, then, for $k,k'= 1,\cdots,m,$ we have
\begin{align*}
P(X_i=k, X_j=k)&=p_k(p_k+c_k|j-i|^{2H_k-2}) , \\P(X_i=k, X_j=k')&=p_k p_{k'},   \end{align*} 
  \begin{align*}
P(X_i=0, X_j=0 )&=1-2\sum_{k=1,\cdots, m}P(X_i=k)+\sum_{k,k'=1,\cdots,m}P(X_i=k,X_j=k')\nonumber\\&=1-2\sum_{k=1}^{m} p_k+\sum_{k=1}^{m}p_k(p_1+p_2+\cdots+p_m+c_k|i-j|^{2H_{k}-2} )\nonumber\\&=p_0^2+\sum_{k=1}^{m}p_kc_k|i-j|^{2H_{k}-2},\\P(X_i=k, X_j=0 )&=P(X_i=0, X_j=k)=p_k(1-p_1-p_2-\cdots -p_m-c_k|i-j|^{2H_k-2}) \nonumber \\&=p_k(p_0-c_k|i-j|^{2H
_k-2}).
\end{align*}

As a result, for $i\neq j, i,j\in \mathbb{N}, k\neq k', k,k'\in \{1,2,\cdots,m\},$
\begin{align}
    &cov(I_{\{X_i=k\}}, I_{\{X_j=k\}})=p_kc_k|i-j|^{2H_k-2},\\
  &cov(I_{\{X_i=k\}}, I_{\{X_j=k'\}})=0,\\
  &cov(I_{\{X_i=0\}}, I_{\{X_j=0\}})=\sum_{k=1}^m p_kc_k|i-j|^{2H_k-2},\\
  &cov(I_{\{X_i=k\}}, I_{\{X_j=0\}})=-p_kc_k|i-j|^{2H_k-2}.
  \end{align} 
Note that $(\{I_{\{X_i=1\}}\}_{i\in \mathbb{N}}, \{I_{\{X_i=2\}}\}_{i\in \mathbb{N}}, \cdots,\{I_{\{X_i=m\}}\}_{i\in \mathbb{N}})$ are $m$ generalized Bernoulli processes with Hurst parameter, $H_1, H_2, \cdots, H_m$, respectively. (see \cite{LEE}). However, they are not independent, since for $\ell\neq k, \ell\in \{1,2,\cdots, m\},$
\[P(\{I_{\{X_i=\ell \}}=1\} \cap \{I_{\{X_i=k \}}=1\} )=0 \neq P(I_{\{X_i=\ell \}}=1)P(I_{\{X_i=k \}}=1)=p_{\ell}p_{k}.\]

Also, we have
   \begin{align*}
cov(X_i, X_j)&=E(X_iX_j)-E(X_i)E(X_j)\\&=\sum_{k,k'}kk'P(I_{\{X_i=k\}}=1, I_{\{X_j=k'\}}=1   )-\sum_{k,k'}kk'p_kp_{k'}\\&=\sum_{k=1,\cdots,m}k^2p_kc_k|i-j|^{2H_k-2}.\end{align*}
Therefore, the process $\{X_i\}_{i\in \mathbb{N}}$ possesses long-range dependence if $\min\{H_1,\cdots, H_k\}>1/2.$

All the results appear in this paper are valid regardless of how the finite-state space of $X_i$ is defined. More specifically, given that  
${\bf D}^*(A_1,A_2,\cdots,A_m;A_0)
>0$ for any pairwise disjoint sets $A_0, A_1,\cdots A_m \subset \mathbb{N},$ we can define probability (2.2-2.4) with any state space $S=\{s_0,s_1, s_2, \cdots, s_m \} \subset \mathbb{R}$  for any $m \in \mathbb{N}$ in the following way.
\begin{align*}
    &P( \cap_{i\in A_k} \{ X_i=s_k\})= L^*_k(A_k) , \text{ for } k=1,\cdots,m,\\ &
P(\cap_{k=1,\cdots,m} \cap_{i\in A_k} \{ X_i=s_k\})=\prod_{k=1,\cdots, m } L^*_k(A_k),\\&
P( \cap_{k=0,\cdots,m} \cap_{i\in A_k} \{ X_i=s_k\})={\bf D}^*(A_1,A_2,\cdots,A_m;A_0).\end{align*}
Note that the only difference is that the space ``$k"$ is replaced by ``$s_k".$ As a result, we can obtain the same results as (2.5-2.8) except that $I_{\{X_i=k\}}$ is replaced by $I_{\{X_i=s_k\}},$ and we get
\begin{align*}
cov(X_i, X_j)&=cov(X_i-s_0, X_j-s_0)\\&=\sum_{k,k'=1,\cdots,m}s_ks_k'P(I_{\{X_i=s_k\}}=1, I_{\{X_j=s_k'\}}=1   )-\sum_{k,k'=1,\cdots,m}s_ks_k'p_kp_{k'}\\&=\sum_{k=1,\cdots,m}(s_k-s_0)^2p_kc_k|i-j|^{2H_k-2}.\end{align*}
In a similar way, all the results in this paper can be easily transfered to any finite-state space $S \subset \mathbb{R}.$ For the sake of simplicity, we assume  $S=\{0,1,\cdots, m\}, m\in \mathbb{N},$ without loss of generality, and define $S^0:=\{1,\cdots, m\}$.

Now we will give a restriction on the parameter values, $\{H_k, p_k, c_k; k\in S^0\}$, which will make ${\bf D}^*(A_1,A_2,\cdots,A_m;A_0)
>0$ for any pairwise disjoint sets $A_0,\cdots A_m \subset \mathbb{N},$
therefore, the process $\{X_i\}$ is well-defined with the probability (2.2-2.4).

 ASSUMPTION:\\  
(a) $c_k,H_k,p_k \in (0,1)$ for $ k\in S^0. $
\\(b) For any  $i_0<i_1<i_2$, $i_0,i_1,i_2\in\mathbb{N},$
\begin{equation}
 \sum_{k=1}^{m}\frac{(p_k+c_k|i_1-i_0|^{2H_k-2})(p_k+c_k|i_2-i_1|^{2H_k-2})}{p_k+c_k|i_2-i_0|^{2H_k-2}}<1.\end{equation}
 
For the rest of the paper, it is assumed that ASSUMPTION (a, b) holds.

\begin{rem}
(a). (2.9) holds if \[\sum_{k=1}^{m}\frac{(p_k+c_k)(p_k+c_k)}{p_k+c_k2^{2H_k-2}}<1,\] since \[\frac{(p_k+c_k|i_1-i_0|^{2H_k-2})(p_k+c_k|i_2-i_1|^{2H_k-2}) }{(p_k+c_k|i_2-i_0|^{2H_k-2})}\] is maximized when $i_2-i_0=2, i_1-i_0=1,$ as it was seen in Lemma 2.5 of  \cite{LEE}.\\
(b). If $({i_{1}-i_{0}})/({i_{2}-i_{0}})\rightarrow 0, ({i_{2}-i_{1}})/({i_{2}-i_{0}})\rightarrow 1$ with $i_{2}-i_{0}\rightarrow \infty$ in (2.9),
then we have
\begin{equation}
\sum_{k=1}^{m} p_k+c_k|i_{1}-i_{0} |^{2H_k-2}   <1,\end{equation}
and this, together with (2.9), implies that for any set $\{A_k, i_k'\} \subset \mathbb{N}, $ \[\sum_{k=1}^{m}\frac{L_{k }^*(A_k \cup\{ i_k'\} )}{L_{k }^*(A_k)}  <1.\] This means that for any $A_0=\{i_0\}\subset \mathbb{N},$  ${\bf D}^*(A_1,A_2,\cdots,A_m;A_0)
>0 $   by (2.1).
\\(c). From (2.10),  $\sum_{k=1}^m c_k<1-\sum_{k=1}^{m}p_k=p_0.$
\\(d). If $m=1,$ (2.9) is reduced to (2.20) in the Lemma 2.5 in \cite{LEE}.
\end{rem}

Now we are ready to show that 
$\{X_i, i\in \mathbb{N}\}$ is well defined with probability (2.2-2.4).

\begin{prop}
For any disjoint sets $ A_0,A_1,A_2,\cdots,A_m \subset \mathbb{N}, A_0\neq \emptyset,$
\[{\bf D}^*(A_1,A_2,\cdots,A_m;A_0)>0. \]

\end{prop}
The next theorem shows that stochastic process
  $\{X_i, i\in \mathbb{N}\}$  defined with probability (2.2-2.4) is stationary, and it has long-range dependence if $\max\{H_k, k\in S^0\}> 1/2.$ Also, indicator sequence of each state is stationary, and it has long-range dependence if its Hurst exponent is greater than 1/2.
\begin{thm}
$\{X_i, i\in \mathbb{N}\}$ is a stationary process with the following properties.\\
i.  \[ P(X_i=k)=p_k, \text{    for }k\in S^0. \] 
ii. 
\[  cov(I_{\{X_i=k}\}, I_{\{X_j=k\}})=p_k c_k |i-j|^{2H_k-2}, \text{    for }k\in S^0,\] and 
\[  cov(I_{\{X_i=0\}}, I_{\{X_j=0\}})\sim p_{k'} c_{k'} |i-j|^{2H_{k'}-2}, \text{    as } |i-j|\rightarrow \infty\]
where $k'=argmax_k H_k.$ \\
iii. \[cov(X_i, X_j)=\sum_{k=1}^m k^2 p_kc_k|i-j|^{2H_k-2}, \text{    for } i\neq j.\]
\end{thm}
\begin{proof}
By Proposition 2.2, $\{X_i\}$ is well defined stationary process with probability (2.2-2.4). The other results follow by (2.5-2.8).
\end{proof}
\section{Tail behavior of inter-arrival time and other properties}

For $k\in S^0,$ $\{I_{\{X_i=k\}}\}_{i\in \mathbb{N}}$ is stationary process in which the event $\{X_i=k\} $   is recurrent, persistent, and aperiodic (Here, we follow  the terminology and definition in \cite{fel}). Define a random variable $T_{kk}^i $ as the inter-arrival time between the i-th $``k"$ from the previous $``k"$, i.e. \[T_{kk}^i:=\inf \{i>0: X_{i+T_{kk}^{i-1}}=k \} ,\]
with $T_{kk}^0:=0.$
Since 
 $\{I_{\{X_i=k\}}\}_{i\in \mathbb{N}}$ is GBP with parameters $(H_k, p_k, c_k)$  for $k\in S^0,$  $T_{kk}^2, T_{kk}^3,\cdots$ are iid (see page 9 \cite{Lee2}). Therefore, we will denote the inter-arrival time between two consecutive observations of $k$ as $T_{kk}.$ The next Lemma is directly obtained from Theorem 3.6 in \cite{Lee2}.

\begin{lem}
For $k\in S^0,$
the inter-arrival time for state $k$, $T_{kk},$ satisfies the following. \\ {\it i.} $T_{kk}$  has mean of $1/p_k$. It has infinite second moment if $H_k \in (1/2,1).$
\\ {\it ii.}
\[ P(T_{kk}>t)=t^{2H_k-3} L_k(t), \]  where $L_k$ is a slowly varying function that depends on the parameter $H_k, p_k, c_k$.
\end{lem}
The first result {\it i} in Lemma 3.1 is similar to Lemma 1 in \cite{CAR2}. But here, we have finite-state stationary process, whereas countable-state space Markov chain was assumed in \cite{CAR2}. 
Now, we investigate the conditional probabilities and the uniqueness of our process.
\begin{thm}
Let $A_0, A_1,\cdots,A_m$ be disjoint subsets of $\mathbb{N}.$ For any $\ell\in S^0$ such that  $\max A_{\ell}>\max A_0,$ and
for $i' \notin \cup_{k=0}^{m} A_k$ such that $i' >\max A_{\ell},$ the conditional probability satisfies the following:
 \begin{align*}
P(X_{i'}=\ell| \cap_{k=0,\cdots,m} \cap_{i\in A_k} \{ X_i=k\})
=p_{\ell}+c_{\ell}|i'-\max A_{\ell}|^{2H_{\ell}-2}.
 \end{align*} If there has been no interruption of ``0" after the last observation of ``$\ell$", then the chance to observe ``$\ell$" depends on the distance between the current time and the last time of observation of ``$\ell$", regardless of how other states appeared in the past. This can be considered as a generalized Markov property. Moreover, this chance to observe $``\ell"$ decreases as the distance increases, following the power law with exponent $2H_{\ell}-2$.  
\begin{proof}
The result follows from the fact that \begin{align*}
P(\{X_{i'}=\ell\}  \cap_{\substack{i\in A_k\\ k\in S^0}} \{ X_i=k\})
=P(  \cap_{i\in A_k, k\in S^0} \{ X_i=k\})\times(p_{\ell}+c_{\ell}|i'-\max A_{\ell}|^{2H_{\ell}-2}),
 \end{align*}  since there is no $i \in A_0$ between $i'$ and $\max A_{\ell}.$

\end{proof}
\end{thm}

In a countable state space Markov chain, long-range dependence is possible only when it has infinite-state space, and additionally if it is stationary,
positive recurrent, irreducible, aperiodic Markov chain, then each state should have the same long-term memory, i.e., sequence indicator have the same Hurst exponent for all states \cite{CAR2}. By relaxing Markov property, long-range dependence was made possible in a finite-state stationary process, also with different Hurst parameter for different states.

\begin{thm}
Let $A_0, A_1, \cdots, A_m$ be disjoint subsets of $\mathbb{N}.$
 For $\ell \in S^0$ such that $\max A_{\ell }<\max A_0$,
and  $i_1',i_2',i_3' \notin \cup_{k=0}^{m}A_k$ such that $i_1',i_2',i_3'>\max A_0,$ and $ i_2'>i_3',$ the conditional probability satisfies the following:
\\ a. 
\begin{equation*}
  p_{\ell}+ c_{\ell} |i_1'-\max A_{\ell}|^{2H_{\ell}-2}> P( X_{i_1'}=\ell|  \cap_{i\in A_k, k\in S^0} \{ X_i=k\} ).\end{equation*}
b. 
\begin{equation*}
   \frac{ P( X_{i_2'}=\ell|  \cap_{i\in A_k, k\in S^0} \{ X_i=k\})
}{    P( X_{i_3'}=\ell|  \cap_{i\in A_k, k\in S^0} \{ X_i=k\} )} >\frac{p_{\ell}+c_{\ell}|i_2'-\max A_{\ell}|^{2H_{\ell}-2}}{p_{\ell}+c_{\ell}|i_3'-\max A_{\ell}|^{2H_{\ell}-2}} \end{equation*}

\end{thm}

\begin{thm}
A stationary process with (2.2-2.4) is the unique stationary process that satisfies\\
i. for $k\in S$ \[P(X_i=k)=p_k, \hspace{10pt} \text{ where } p_k >0 \text{ and } \sum_{k=0}^{m} p_k=1, \]\\
ii. for   $k\in S^0$ and any $ i,j \in \mathbb{N}, i\neq j,$  \[Cov(I_{\{X_i=k\}}, I_{\{X_j=k\}})=c'_k|i-j|^{2H_k-2}, \]
 for some constants
 $c'_k\in \mathbb{R}_+, H_k \in(0,1),$
\\ iii. for any sets, $A\subset S^0$ and $\{i_k; k\in A\}\subset \mathbb{N},$
\[P(\cap_{k\in A} \{X_{i_k}=k\})=\prod_{k\in A} p_k,\]
\\
iv. for $\ell\in S^0,$ there is a function $h_{\ell}(\cdot)$ such that
\begin{align*}
P(X_{i'}=\ell|  \cap_{i\in A_k, k\in S^0} \{ X_i=k\})
=h_{\ell}(i'-\max A_{\ell})
 \end{align*}
for disjoint subsets,  $A_0, A_1,\cdots,A_m, \{i'\} \subset \mathbb{N}$, such that  $A_{\ell}\neq \emptyset,$ $i' >\max A_{\ell},$ and  $ \max A_{\ell}>\max A_0$ ($A_0$ can be the empty set).

\end{thm}
\begin{proof}
Let $X^*$ be a stationary process that satisfies $i-iv$. By $i,ii,$ \[P(X^*_{i_0}=k, X^*_{i_1}=k)=Cov(I_{\{X^*_{i_0}=k\}},I_{\{X^*_{i_1}=k\}})+p_k^2=c_k'|i_0-i_1|^{2H_k-2}+p_k^2,\] which results in 
 \[h_k(i_0-i_1)=P(X^*_{i_1}=k|X^*_{i_0}=k)= p_k+(c_k'/p_k)|i_0-i_1|^{2H_k-2}. \] Therefore, by $iv,$ \begin{align*}
     P(X^*_{i_0}=k, X^*_{i_1}=k,X^*_{i_2}=k, \cdots, X^*_{i_n}=k )&=p_k\prod_{j=1}^{n}h_k(i_j-i_{j-1}) \\&=L^*_{k}(\{i_0,i_2,\cdots, i_n\}) ,\end{align*} where $L_k^*=L_{H_k,p_k,c_k'/p_k}^*.$
 Also, by applying {\it iii,iv} to $X^*$, 
 \begin{align*}
     P( \cap_{i\in A_k, k\in S^0} \{ X_i=k\})=\prod_{k=1,\cdots, m } L_{k}^*(A_k).\end{align*}
 This implies that $X^*$ satisfies (2.2-2.4) with $c_k=c_k'/p_k$ for  $k\in S^0.$
\end{proof}
\section{Fractional multinomial distribution}
In this section, we define fractional multinomial distribution that can serve as a over-dispersed multinomial distribution.

Note that $\sum_{i=1}^{n}I_{\{X_i=k\}}$
has mean $np_k$ for $k\in S.$  Also, as $n\rightarrow \infty, $  
\[ var\Big(\sum_{i=1}^{n}I_{\{X_i=k\}}\Big)\sim  \begin{dcases}  (p_k(1-p_k)+\frac{c'_k}{2H_k-1})n&H_k\in (0, 1/2),  \\   {c'_k} {n}{\ln n } &H_k=1/2,\\ \frac{c'_k}{2H_k-1}|n|^{2H_k}, &H_k\in (1/2, 1) ,   \end{dcases}\] for $k\in S^0$, and

\[ var\Big(\sum_{i=1}^{n}I_{\{X_i=0\}}\Big)\sim  \begin{dcases}  (p_{k'}(1-p_{k'})+\frac{c'_{k'}}{2H_{k'}-1})n&H_{k'}\in (0, 1/2),  \\   {c'_{k'}} n{\ln n } &H_{k'}=1/2,\\ \frac{c'_{k'}}{2H_{k'}-1}|n|^{2H_{k'}}, &H_{k'}\in (1/2, 1) ,  \end{dcases}\] 
where $k'=argmax_k \{H_k;k\in S^0\}$, and $c'_k=p_kc_k.$
It also has the following covariance.
\[ cov\Big(\sum_{i=1}^{n}I_{\{X_i=k\}},\sum_{i=1}^{n}I_{\{X_i=k'\}}\Big)=-np_kp_{k'},\] 
\[ cov\Big(\sum_{i=1}^{n}I_{\{X_i=0\}},\sum_{i=1}^{n}I_{\{X_i=k\}}\Big)=-np_0p_{k}-\sum_{\substack{ i\neq j \\i,j=1,\cdots, n }}c'_k|i-j|^{2H_k-2},\]
for $k, k' \in S^0.$

We define $Y_k:=\sum_{i=1}^n I_{\{X_i=k\}} ,$ for $k\in S,$ and a fixed $n$, and call its distribution fractional multinomial distribution with parameters $n, {\bf p}, {\bf H}, {\bf c}.$

If ${\bf c}={\bf 0}$, $(Y_0, Y_1,Y_2,\cdots,Y_m )$ follows multinomial distribution with parameters $n,{\bf p},$ and
$E(Y_k)=np_k, var(Y_k)=np_k(1-p_k), cov(Y_k, Y_{k'})=-np_kp_{k'},$ for $k,k'\in S, k\neq k',$ and $p_0=1-\sum_{i=1}^m p_i.$

If ${\bf c}\neq {\bf  0},$ $(Y_0,Y_1,\cdots,Y_m)$ can serve as  over-dispersed multinomial random variables with \[ E(Y_k)=np_k, \hspace{10pt} Var(Y_k)=np_k(1-p_k)(1+\psi_{n,k}) ,\]  where the over-dispersion parameter $\psi_{n,k}$ is as follows. \[\psi_{n,k}\sim\begin{dcases}  \frac{c}{(1-p_k)(2H_k-1)} & \text{ if } H_k\in (0,1/2),\\
\frac{c \ln n}{1-p_k}-1 &  \text{ if } H_k=1/2,\\ \frac{c n^{2H_k-1}}{(1-p_k)2H_k-1}-1& \text{ if } H_k\in (1/2,1),\end{dcases}\]
for $k\in S^0,$ and 
 \[\psi_{n,0}\sim\begin{dcases}  \frac{c}{(1-p_{k'})(2H_{k'}-1)} & \text{ if } H_{k'}\in (0,1/2),\\
\frac{c \ln n}{1-p_{k'}}-1 & \text{ if } H_{k'}=1/2,\\ \frac{c n^{2H_{k'}-1}}{(1-p_{k'})2H_{k'}-1}-1& \text{ if }H_{k'}\in (1/2,1),\end{dcases}\] where $k'=argmax_k \{ H_k; k\in S^0\},$
as $n\rightarrow \infty.$
If $H_k\in (0,1/2),$ the over-dispersion parameter $\psi_{n,k}$ remains stable as $n$ increases, whereas if $H_k\in (1/2,1)$ the over-dispersed parameter $\psi_{n,k}$ increases with the rate of fractional exponent of $n$, $n^{2H_k-1}.$ 

\section{Conclusion}
A new method for modeling long-range dependence in discrete-time finite-state stationary process was proposed. This model allows different states to have different Hurst indices except that for the base state ``0" the Hurst exponent is the maximum Hurst index of all other states. Inter-arrival time for each state follows heavy tail distribution, and its tail behavior is different for different states.  The other interesting feature of this process is that  the conditional probability to observe a state ``$k$" ($k$ is not the base state ``0")  depends on the Hurst index $H_k$ and the time difference between the last observation of ``$k$" and the current time, no matter how other states appeared in the past, given that there was no base state observed since the last observation of ``$k$".
From the stationary process developed in this paper, we defined fractional multinomial distribution which can express a wide range of  over-dispersed multinomial distribution; each state can have different over-dispersion parameter that can behave asymptotically constant or grow with a fractional exponent of the number of trials.


\section{Proofs}

\begin{lem} For any $\{a_0, a_1, \cdots, a_n, a_0',a_1',\cdots, a_n'\} \subset \mathbb{R}_{+}$ that satisfies
$a_0-\sum_{i=1}^{j}a_i>0, a_0'-\sum_{i=1}^{j}a_i'>0$ for $j=1,2,\cdots,n$,
\\
i. if
\[\frac{a_0}{a_0'} \geq \frac{a_1}{a_1'}\geq\cdots\geq\frac{a_n}{a_n'}, \] then \[ \frac{a_0-a_1-a_2-\cdots -a_n}{a_0'-a_1'-a_2'-\cdots -a_n'}\geq \frac{a_0}{a_0'}.\]
ii. If \[\frac{a_0}{a_0'} <\frac{a_1}{a_1'}\leq\cdots\leq\frac{a_n}{a_n'}, \] then \[ \frac{a_0-a_1-a_2-\cdots -a_n}{a_0'-a_1'-a_2'-\cdots -a_n'}\leq \frac{a_0}{a_0'}.\]
iii.   For any $\{a_0, a_1, \cdots, a_n, a_0',a_1',\cdots, a_n'\} \subset \mathbb{R}_{+}$ 
\[ \max_i \frac{a_i}{a_i'}\geq \frac{a_1+a_2+\cdots+a_n}{a_1'+a_2'+\cdots+a_n'}\geq \min_i \frac{a_i}{a_i'} .\]

\end{lem}

\begin{proof}

{\it i} and  {\it ii} were proved in Lemma 5.2 in \cite{LEE}.
\\For {\it iii}, 
define $b_j$ such that  \[\frac{a_j}{a_j'}= b_j.\] Then 
\[\frac{a_1+a_2+\cdots+a_n}{a_1'+a_2'+\cdots+a_n'}= \frac{b_1a_1'+b_2a_2'+\cdots+b_na_n'}{a_1'+a_2'+\cdots+a_n'}  \] which is weighted average of $\{b_j, j=1,\cdots,n\}$. 
\end{proof}

 To ease our notation, we will denote \[{\bf L}^*(A_1,A_2,\cdots,A_{k-1}, A_k\cup\{i\}, A_{k+1}, \cdots,A_m)\]  by  \[{\bf L}^*(\cdots, A_k\cup\{i\}, \cdots) ,\] and 
 
 \[ {\bf L}^*(\cdots, A_k\cup\{i\},   A_{k'}\cup\{j\},\cdots )=  {\bf L}^*(A_1^*,A_2^*,\cdots, A_m^*)  \] where if $k\neq k'$
 \[ A_i^*=\begin{dcases}
 A_i \text{ if } i \neq k, k,'\\
 A_i\cup \{i\}   \text{ if } i = k,\\
 A_i\cup \{j\}   \text{ if } i = k',
 \end{dcases}\]
 and if $k=k'$ \[ A_i^*=\begin{dcases}
 A_i \text{ if } i \neq k, \\
 A_i\cup \{i\cup\}   \text{ if } i = k.
 \end{dcases}\]
${\bf D}^*(\cdots, A_k\cup\{i\}, \cdots)$ and $ {\bf D}^*(\cdots, A_k\cup\{i\},   A_{k'}\cup\{j\},\cdots )$ are also defined in a similar way. 
\begin{lem} For any disjoint sets $A_1,\cdots, A_m, \{i_0,i_1\}\subset \mathbb{N}$,\\
i. \[ {\bf D}^*(A_1,A_2,\cdots,A_m;\{i_0\})>0 \]
\\ ii.  \[ {\bf D}^*(A_1,A_2,\cdots,A_m;\{i_0,i_1\})>0 \]

\end{lem}
\begin{proof}
{\it i.} \begin{align*}
{\bf D}^*(A_1,A_2,\cdots,A_m;\{i_0\})&=\prod_{k=1}^{m}L_{k}^*(A_k) \bigg(1-\sum_{k'=1}^{m} \frac{  L_{k'}^*(A_{k'}\cup\{i_0\} ) }{L_{k'}^*(A_{k'})} \bigg)\\&=\prod_{k=1}^{m}L_{k}^*(A_k)\bigg (1-\sum_{k'=1}^{m} \frac{L_{k'}^*(\{i_{1,k'},i_{2,k'},i_0\})}{L_{k'}^*(\{i_{1,k'},i_{2,k'}\})} \bigg)
\end{align*}
where $i_{1,k'}, i_{2,k'} \in A_{k'}$ are two closest elements to $i_0$ among $A_{k'}$ such that if $\min A_{k'} <i_0< \max A_{k'}  ,$ then $  i_{1,k'} <i_0<i_{2,k'}, $ if $i_0>\max A_{k'},$ then  
$  i_{1,k'}<i_{2,k'}<i_0, $  if $i_0<\min A_{k'}, $ then $ i_0<i_{1,k'}<i_{2,k'},  $ and if  $A_{k'}=\emptyset,$ then  $i_{1,k'}=i_{2,k'}=\emptyset$. Therefore,
\begin{align}
&\frac{L_{k'}^*(\{i_{1,k'},i_{2,k'},i_0\})}{L_{k'}^*(\{i_{1,k'},i_{2,k'}\})} \nonumber \\&=\begin{dcases} \frac{ (p_{k'}+ c_{k'}|  i_{1,k'}-i_0|^{2H_{k'}-2}) (p_{k'}+c_{k'}|i_0-i_{2,k'} |^{2H_{k'}-2}) }{   p_{k'}+c_{k'}|i_{1,k'}-i_{2,k'}|^{2H_{k'}-2}} &\text{ if }  \min A_{k'} <i_0< \max A_{k'},\\
p_{k'}+c_{k'}|\max A_{k'}-i_0 |^{2H_{k'}-2}  &\text{ if } i_0>\max  A_{k'},\\
p_{k'}+c_{k'}|\min A_{k'}-i_0 |^{2H_{k'}-2} &\text{ if } i_0<\min A_{k'},\\
p_{k'} &\text{ if } A_{k'}=\emptyset. \nonumber
\end{dcases} \end{align} 
By (2.9), $\sum_{k'=1}^{m} \frac{L_{k'}^*(\{i_{1,k'},i_{2,k'},i_0\})}{L_{k'}^*(\{i_{1,k'},i_{2,k'}\})}<1,$ and the result is derived.
\\ {\it ii.} Since
\begin{align*}
 {\bf D}^*(A_1,A_2,\cdots,A_m;\{i_0,i_1\})&=
{\bf D}^*(A_1,A_2,\cdots,A_m;\{i_0\})-\sum_{k=1}^{m}{\bf D}^*(\cdots,A_k\cup\{i_1\}, \cdots ;\{i_0\}),\end{align*}
it is sufficient if we show 
\begin{align*}\frac{{\bf L}^*(A_1,A_2,\cdots ,A_m)-\sum_{k=1}^{m}{\bf L}^*(\cdots, A_k\cup\{i_0\},\cdots )}{  \sum_{k'=1}^{m}  {\bf L}^*(\cdots ,A_{k'}\cup \{i_1\},\cdots)-\sum_{k,k'=1}^{m}{\bf L}^*(\cdots, A_k\cup\{i_0\},A_{k'}\cup \{i_1\},\cdots )  }>1.
\end{align*}
Note that 
\begin{align*}
\frac{{\bf L}^*(A_1,A_2,\cdots ,A_m)}{ \sum_{k'=1}^{m} {\bf L}^*(\cdots ,A_{k'}\cup \{i_1\},\cdots)}=\frac{1}{\sum_{k'=1}^{m} \frac{L_{k'}^*(\{i_{1,k'},i_{2,k'},i_0\})}{L_{k'}^*(\{i_{1,k'},i_{2,k'}\})}},
\end{align*}
which is non-increasing as set $A_k$ increases for $k=1,\cdots, m$. That is, 
\[\frac{{\bf L}^*(A_1,A_2,\cdots ,A_m)}{ \sum_{k'=1}^{m}  {\bf L}^*(\cdots ,A_{k'}\cup \{i_1\},\cdots)} \leq \frac{{\bf L}^*(A_1',A_2',\cdots ,A_m')}{ \sum_{k'=1}^{m}  {\bf L}^*(\cdots ,A'_{k'}\cup \{i_1\},\cdots)}\]
for any sets $ A
_k\subseteq A'_k, k=1,2,\cdots, m.$
Therefore,
\begin{align*}
\frac{{\bf L}^*(A_1,A_2,\cdots ,A_m)}{ \sum_{k'=1}^{m} {\bf L}^*(\cdots ,A_{k'}\cup \{i_1\},\cdots)}>\frac{\sum_{k=1}^{m}{\bf L}^*(\cdots, A_k\cup\{i_0\},\cdots )}{\sum_{k,k'=1}^{m}{\bf L}^*(\cdots, A_k\cup\{i_0\},A_{k'}\cup \{i_1\},\cdots ) }
\end{align*}
 by Lemma 6.1 {\it iii}. By Lemma 6.1 {\it i} combined with the fact that \[    \frac{1}{\sum_{k'=1}^{m} \frac{L_{k'}^*(\{i_{1,k'},i_{2,k'},i_0\})}{L_{k'}^*(\{i_{1,k'},i_{2,k'}\})}}>1\] from (2.9), the result is derived.
\end{proof}

Note that for any disjoint sets $A_1, A_2, \cdots, A_m, \{i_0,i_1,\cdots , i_n\}$
\begin{align*}    {\bf D}^*(A_1,A_2,\cdots,A_m;\{i_0,i_1,\cdots,i_n\})&= {\bf D}^*(A_1,A_2,\cdots,A_m;\{i_0,i_1,\cdots,i_{n-1}\})\\&- {\bf D}^*(A_1\cup \{ i_n\},A_2,\cdots,A_m;\{i_0, i_1, \cdots, i_{n-1}\})  
\\& - {\bf D}^*(A_1,A_2\cup \{ i_n\},\cdots,A_m;\{i_0, i_1, \cdots, i_{n-1}\})  
\\& \hdots 
\\& - {\bf D}^*(A_1,A_2,\cdots,A_m\cup \{ i_n\} ;\{i_0, i_1, \cdots, i_{n-1}\}).  
\end{align*}

Let's  denote \begin{align*}
\sum_{k=1}^{m}{\bf D}^*(A_1, \cdots,A_{k-1},A_k\cup \{ i_n\}, A_{k+1}\cdots,A_m ;\{i_0, i_1, \cdots, i_{n-1}\}) 
\end{align*}
by 
\[\sum_{k=1}^{m}{\bf D}^*(\cdots,A_k\cup \{ i_n\},\cdots;\{i_0, i_1, \cdots, i_{n-1}\}).\]
\\
\\

{\textit{Proof of Proposition 2.2.}}\\
We will show by mathematical induction that $\{X_{i_1},\cdots,X_{i_n}\}$ is a random vector with probability (2.2-2.4) for any $n$ and any $\{i_1, i_2, \cdots,i_n \} \subset \mathbb{N}$. For $n=1,$ it is trivial. For $n=2,$ it is proved by Lemma 6.2. Let's assume that $\{X_{i_1},\cdots,X_{i_{n'-1}}\}$ is a random vector with probability (2.2-2.4) for any $ \{i_1, i_2,\cdots, i_{n'-1}\}\subset \mathbb{N}$. We will prove that $\{X_{i_1},\cdots,X_{i_{n'}}\}$
is a random vector for any $\{i_1,i_2,\cdots, i_{n'}\}\subset \mathbb{N}.$   

Without loss of generality, fix a set $\{i_1, i_2,\cdots, i_{n'}\} \subset \mathbb{N}.$
 To prove that $\{X_{i_1},\cdots, X_{i_{n'}}\}$ is a random vector with probability (2.2-2.4), we need to show that ${\bf D}^*(A_1, \cdots, A_m;A_0)>0$ for any pairwise disjoint sets, $A_0,\cdots, A_m,$  such that 
$\cup_{k=0}^{m} A_k =\{{i_1},\cdots,{i_{n'}}\}.$ If $|A_0|=0 $ or 1, then the result follows from the definition of ${\bf D}^*$ and Lemma 6.2, respectively.  Therefore, we assume that $|A_0|\geq 2,$ $A_0=\{i'_0,i'_1,\cdots,i'_{n_0}\},$  and $\max A_0=i'_{n_0}.$ Let $A'_0=A_0/\{i'_{n_0}\}.$ We will first show that for any such sets, 
\begin{equation}
   \frac{{\bf D}^*(A_1, \cdots, A_m;A'_0)}{\sum_{\ell=1}^{m}{\bf D}^*(\cdots, A_{\ell}\cup\{i'_{n_0}\},\cdots;A'_0)}>1.\end{equation}
 (6.1) is equivalent to ${\bf D}^*(A_1, \cdots, A_m;A_0)>0.$

For fixed $\ell \in \{1,2,\cdots,m\},$
define the following vectors of length $m-1,$
\begin{align*}
{\bf H}^{\ell}&=(H_1,\cdots,H_{\ell-1},H_{\ell+1},\cdots,H_m),\\ {\bf p}^{\ell}&=(p_1,\cdots,p_{\ell-1},p_{\ell+1},\cdots,p_m),\\
{\bf c}^{\ell}&=(c_1,\cdots,c_{\ell-1},c_{\ell+1},\cdots,c_m)
.\end{align*}  
We also define 
${\bf D}_{(-\ell)} ^*(\cdots, A_{\ell-1},A_{\ell+1},\cdots;A_0):=D_{{\bf H}^{\ell}, {\bf p}^{\ell}, {\bf c}^{\ell}}^*(A_1,  \cdots, A_{\ell-1},A_{\ell+1},\cdots, A_m;A_0). $

Since $\{X_{i}; i\in \cup_{k=1}^m A_k \cup A'_0\} $ is a random vector with (2.2-2.4),
${\bf D}^*(\cdots, A_{\ell},\cdots;A'_0)>0,$ and it can be written as
\begin{align}
{\bf D}^*(\cdots, A_{\ell},,\cdots;A'_0)=P\Big(\cap_{i\in A'_0}\{X_i=0\} \cap_{\substack{i\in A_k\\k=1,\cdots,m\\ k\neq \ell}}\{X_i=k\} \cap_{i\in A_{\ell} }\{X_i=\ell\} \Big)\end{align}\begin{align}&
=P\Big(\cap_{i\in A'_0}\{X_i\in\{0,\ell\}\} \cap_{\substack{i\in A_k\\k=1,\cdots,m\\ k\neq \ell}}\{X_i=k\} \cap_{i\in A_{\ell} }\{X_i=\ell\}\Big )
\nonumber\\&-P\Big(\cap_{i\in A'_0/\{i'_0\}}\{X_i\in\{0,\ell\}\} \cap_{\substack{i\in A_k\\k=1,\cdots,m\\ k\neq \ell}}\{X_i=k\} \cap_{i\in A_{\ell}\cup\{i'_0\} }\{X_i=\ell\} \Big)
\nonumber \\&-P\Big(\cap_{i\in A'_0/\{i'_0,i'_1\}}\{X_i\in\{0,\ell\}\} \cap_{\substack{i\in A_k\\k=1,\cdots,m\\ k\neq \ell}}\{X_i=k\} \cap_{i\in A_{\ell}\cup\{i'_1\} }\{X_i=\ell\}\cap\{X_{i'_0}=0\} \Big)
\nonumber\\&-P\Big(\cap_{i\in A'_0/\{i'_0,i'_1,i'_2\}}\{X_i\in\{0,\ell\}\} \cap_{\substack{i\in A_k\\k=1,\cdots,m\\ k\neq \ell}}\{X_i=k\} \cap_{i\in A_{\ell}\cup\{i'_2\} }\{X_i=\ell\}\cap_{i\in\{i'_0,i'_1\}}\{X_{i}=0\} \Big)
\nonumber\\& \vdots
\nonumber\\&-P\Big( \cap_{\substack{i\in A_k\\k=1,\cdots,m\\ k\neq \ell}}\{X_i=k\} \cap_{i\in A_{\ell}\cup\{i'_{n_0-1}\} }\{X_i=\ell\}\cap_{i\in A'_0/\{i'_{n_0-1}\}}\{X_{i}=0\} \Big).\nonumber
\end{align}
Note that
\begin{eqnarray}
   \hspace{25pt}& &P\Big(\cap_{i\in A'_0}\{X_i\in\{0,\ell\}\} \cap_{\substack{i\in A_k\\k=1,\cdots,m\\ k\neq \ell}}\{X_i=k\} \cap_{i\in A_{\ell} }\{X_i=\ell\} \Big)\\&&=P\Big(\cap_{\substack{i\in A_k\\k=1,\cdots,m\\ k\neq \ell}}\{X_i=k\} \cap_{i\in A_{\ell} }\{X_i=\ell\}\Big)
\nonumber\\&&-P\Big(\cap_{i\in A'_0}\{X_i\in\{1,\cdots,\ell-1,\ell+1,\cdots,m\}\} \cap_{\substack{i\in A_k\\k=1,\cdots,m\\ k\neq \ell}}\{X_i=k\} \cap_{i\in A_{\ell} }\{X_i=\ell\} \Big)\nonumber\\&&
=L_{{\ell}}^*(A_{\ell}){\bf D}_{(-\ell)}^*(\cdots, A_{\ell-1},A_{\ell+1},\cdots;A'_0),\nonumber\end{eqnarray}
and
\begin{eqnarray}
 P\Big(\cap_{i\in \{i'_{j+1},\cdots,i'_{n_0-1}\}}\{X_i\in\{0,\ell\}\} & & \cap_{\substack{i\in A_k\\ k=1,\cdots,m\\ k\neq \ell}}\{X_i=k\} \\ &&\cap_{i\in A_{\ell}\cup\{i'_j\} }\{X_i=\ell\}\cap_{i\in\{i'_0,\cdots,i'_{j-1}\}}\{X_{i}=0\}\Big ) \nonumber \end{eqnarray}
 \begin{eqnarray}&& =   P\Big(\cap_{i\in A'_0/\{i_j'\}}\{X_i\in\{0,\ell\}\}\cap_{\substack{i\in A_k\\ k=1,\cdots,m\\ k\neq \ell}}\{X_i=k\} \cap_{i\in A_{\ell}\cup\{i'_j\} }\{X_i=\ell\}\Big)\nonumber\\&&-\sum_{i^*\in A'_0, i^*<i_j'}  P\Big(\cap_{i\in A'_0/\{i_j',i^*\}}\{X_i\in\{0,\ell\}\}\cap_{\substack{i\in A_k\\ k=1,\cdots,m\\ k\neq \ell}}\{X_i=k\} \cap_{i\in A_{\ell}\cup\{i'_j,i^*\} }\{X_i=\ell\}\Big)
\nonumber\\
&&+ \sum_{\substack{i^*,i^{**}\in A'_0,\\ i^*<i^{**}<i_j'}}  P\Big(\cap_{i\in A'_0/\{i_j',i^*,i^{**}\}}\{X_i\in\{0,\ell\}\}\cap_{\substack{i\in A_k\\ k=1,\cdots,m\\ k\neq \ell}}\{X_i=k\} \cap_{i\in A_{\ell}\cup\{i'_j,i^*,i^{**}\} }\{X_i=\ell\}\Big)
\nonumber\\
&&-\cdots (-1)^{j}  P\Big(\cap_{i\in A'_0/\{i_j',i'_0, i'_1, \cdots, i'_{j-1}\}}\{X_i\in\{0,\ell\}\}\cap_{\substack{i\in A_k\\ k=1,\cdots,m\\ k\neq \ell}}\{X_i=k\} \cap_{i\in A_{\ell}\cup\{i'_j,i'_0, i'_1, \cdots, i'_{j-1}\} }\{X_i=\ell\}\Big)
\nonumber
\\&& = \sum_{\substack{C\cap D=\emptyset\\ C=\emptyset \text{ or } \max C<i'_j \\ C\cup D =A'_0/\{i'_j\}}} (-1)^{|C|}L_{\ell}^*(   A_{\ell}\cup\{i'_j\}\cup C  )  {\bf D}^*_{(-\ell)}(\cdots, A_{\ell-1},A_{\ell+1},\cdots;D) \nonumber
\end{eqnarray} 
where  $|\emptyset|=0.$
Therefore, by (6.2-6.4),
\begin{eqnarray}
&&\hspace{25pt}{\bf D}^*(\cdots, A_{\ell},\cdots;A'_0)
= L_{\ell}^*(A_{\ell}){\bf D}_{(-\ell)}^*(\cdots, A_{\ell-1},A_{\ell+1},\cdots;A'_0)
\\&& +\sum_{j=0}^{n_0-1} \sum_{\substack{C\cap D=\emptyset\\ C=\emptyset \text{ or } \max C<i'_j \\ C\cup D =A'_0/\{i'_j\}}} (-1)^{|C|+1}L_{\ell}^*(   A_{\ell}\cup\{i'_j\}\cup C  )  {\bf D}^*_{(-\ell)}(\cdots, A_{\ell-1},A_{\ell+1},\cdots;D). \nonumber
\end{eqnarray}  (6.5) can also be derived by the definition of $L_{\ell}^*, {\bf D}^* ,$ without using probability for $\{X_{i}; i\in \cup_{k=1}^m A_k \cup A'_0\}$. In the same way, using the definition of $L_{\ell}^*, {\bf D}^* ,$ 

\begin{eqnarray}
&& \hspace{25pt}{\bf D}^*(\cdots, A_{\ell}\cup\{i'_{n_0}\},\cdots;A'_0)
= L_{\ell}^*(A_{\ell}\cup \{i'_{n_0}\}){\bf D}^*_{(-\ell)}(\cdots, A_{\ell-1},A_{\ell+1},\cdots;A'_0)
\\&& +\sum_{j=0}^{n_0-1} \sum_{\substack{C\cap D=\emptyset\\ C=\emptyset \text{ or } \max C<i'_j \\ C\cup D =A'_0/\{i'_j\}}} (-1)^{|C|+1}L_{\ell}^*(   A_{\ell}\cup\{i'_{n_0},i'_j\}\cup C  )  {\bf D}^*_{(-\ell)}(\cdots, A_{\ell-1},A_{\ell+1},\cdots;D).\nonumber
\end{eqnarray}  

Note that, for $j=0,1,\cdots, n_0-1,$
\begin{align*}&g_{\bf{H}, \bf{p},\bf c}(A_1,\cdots,A_{\ell}\cup\{i'_{n_0}\},\cdots,A_m;A'_0;i'_j):=\\& \sum_{\substack{C\cap D=\emptyset\\ C=\emptyset \text{ or } \max C<i'_j \\ C\cup D =A'_0/\{i'_j\}}} (-1)^{|C|+1}L_{\ell}^*(   A_{\ell}\cup\{i'_{n_0},i'_j\}\cup C  )  {\bf D}^*_{(-\ell)}(\cdots, A_{\ell-1},A_{\ell+1},\cdots;D)<0,\end{align*}
since we have \begin{align*}
    &g_{\bf{H}, \bf{p},\bf c}(A_1,\cdots,A_{\ell},\cdots,A_m;A'_0;i'_j)=\\&-P\Big(\cap_{i\in \{i'_{j+1},\cdots,i'_{n_0-1}\}}\{X_i\in\{0,\ell\}\} \cap_{\substack{i\in A_k\\ k=1,\cdots,m\\ k\neq \ell}}\{X_i=k\} \cap_{i\in A_{\ell}\cup\{i'_j\} }\{X_i=\ell\}\cap_{i\in\{i'_0,\cdots,i'_{j-1}\}}\{X_{i}=0\} \Big)\\&<0
\end{align*}  by (6.4),
and  \begin{equation}
    f_{H_{\ell},p_{\ell}, c_{\ell}}(A_{\ell};i'_j;i'_{n_0}):=\frac{g_{\bf{H}, \bf{p},\bf c}(A_1,\cdots,A_{\ell},\cdots,A_m;A'_0;i'_j)}{g_{\bf{H}, \bf{p},\bf c}(A_1,\cdots,A_{\ell}\cup\{i'_{n_0}\},\cdots,A_m;A'_0;i'_j)}>1 .\end{equation}  The last inequality is due to the fact that 

\begin{align*}
    &\frac{g_{\bf{H}, \bf{p},\bf c}(A_1,\cdots,A_{\ell},\cdots,A_m;A'_0;i'_j)}{g_{\bf{H}, \bf{p},\bf c}(A_1,\cdots,A_{\ell}\cup\{i'_{n_0}\},\cdots,A_m;A'_0;i'_j)}\\&=\frac{   \sum_{j=0}^{n_0-1} \sum_{\substack{C \subseteq A'_0/\{i'_j\} \\C=\emptyset \text{ or } \max C<i'_j  }} (-1)^{|C|+1}L_{\ell}^*(   A_{\ell}\cup\{i'_j\}\cup C  )  }{  \sum_{j=0}^{n_0-1} \sum_{\substack{  C \subseteq A'_0/\{i'_j\} \\C=\emptyset \text{ or } \max C<i'_j  }} (-1)^{|C|+1}L_{\ell}^*(   A_{\ell}\cup\{i'_{n_0},i'_j\}\cup C  )   }, \end{align*} and   for any set $C$ such that $\max C<i'_j$ or $C=\emptyset,$ 
 \[ \frac{L_{\ell}^*(   A_{\ell}\cup \{i'_j\}\cup C  )} {L_{\ell}^*(   A_{\ell}\cup\{i'_{n_0},i'_j\} \cup C  )}=\frac{L_{\ell}^*(   A_{\ell}\cup \{i'_j\}  )} {L_{\ell}^*(   A_{\ell}\cup\{i'_{n_0},i'_j\}   )}>1 \] by (2.9). More specifically,
\begin{equation}
     f_{H_{\ell},p_{\ell}, c_{\ell}}(A_{\ell};i'_j;i'_{n_0})=\frac{L_{\ell}^*(   A_{\ell}\cup \{i'_j\}\cup C  )} {L_{\ell}^*(   A_{\ell}\cup\{i'_{n_0},i'_j\} \cup C  )} =
\frac{L_{\ell}^*(i_{\ell,j,1},i_{\ell,j,2})}{L_{
\ell}^*(i_{\ell,j,1},i_{\ell,j,2}, i'_{n_0})}   \end{equation}
where $i_{\ell,j,1},i_{\ell,j,2}$ are the two closest elements to $i'_{n_0}$ among $A_{\ell}\cup \{i'_j\}$. That is, $i_{\ell,j,1}, i_{\ell,j,2} \in A_{\ell}\cup\{i'_j\}$ are two closest elements to $i'_{n_0}$ such that if $\min A_{\ell}\cup\{i'_j\} <i'_{n_0}< \max A_{\ell} ,$ then $  i_{\ell,j,1} <i'_{n_0}<i_{\ell,j,2}, $ and if $i'_{n_0}>\max A_{\ell}\cup\{i'_j\},$ then  
$  i_{\ell,j,1}<i_{\ell,j,2}<i'_{n_0}. $   
\begin{align}
&\frac{L_{\ell}^*(\{i_{\ell,j,1},i_{\ell,j,2}\})}{L_{\ell}^*(\{i_{\ell,j,1},i_{\ell,j,2},i_{n'}\})} \nonumber \\&=\begin{dcases} \frac{   p_{\ell}+c_{\ell}|i_{\ell,j,1}-i_{\ell,j,2}|^{2H_{\ell}-2}}{ (p_{\ell}+ c_{\ell}|  i_{\ell,j,1}-i_{n'}|^{2H_{\ell}-2}) (p_{\ell}+c_{\ell}|i_{n'}-i_{\ell,j,2} |^{2H_{\ell}-2}) } &\text{ if }  \min A_{\ell}\cup\{i'_j\} <i_{n'}< \max A_{\ell},\\
\frac{1}{p_{\ell}+c_{\ell}|i_{\ell,j,2}-i_{n'} |^{2H_{\ell}-2} } &\text{ if } i_{n'}>\max  A_{\ell}\cup\{i'_j\}, \nonumber
\end{dcases} \end{align} which is non-increasing as $j$ increases since $i'_j<i'_{n_0}.$
 Therefore, $ f_{H_{\ell},p_{\ell}, c_{\ell}}(A_{\ell};i'_j;i'_{n_0})$ is non- increasing as $j$ increases.
Also, for fixed $j, C$ such that $\max C <i'_j$ or $C=\emptyset$,
\begin{align}
\frac{L_{\ell}^*(   A_{\ell}\cup\{i'_{n_0},i'_j\} \cup C  )}{L_{\ell}^*(   A_{\ell}\cup \{i'_j\}\cup C  )}
\geq 
\frac{L_{\ell}^*(   A_{\ell}\cup\{i'_{n_0}\}  )}{L_{\ell}^*(   A_{\ell}  )} 
\end{align}
by the fact that 
$ \frac{L_{\ell}^*(A\cup \{i\})}{L_{\ell}^*(A)} $ is non-decreasing as the set $A$ increases.

Combining the above facts with (6.5-6.6), and by Lemma 6.1 {\it i}, \begin{align*}
 \frac{L_{\ell}^*(A_{\ell})}{L_{\ell}^*(A_{\ell}\cup \{i'_{n_0}\})} \leq \frac{{\bf D}^*(\cdots, A_{\ell},\cdots;A'_0)}{{\bf D}^*(\cdots, A_{\ell}\cup\{i'_{n_0}\},\cdots;A'_0)}.
    \end{align*}

Therefore,
\begin{align*}
 \frac{{\bf D}^*(A_1, \cdots, A_m;A'_0)}{\sum_{\ell=1}^{m}{\bf D}^*(\cdots, A_{\ell}\cup\{i'_{n_0}\},\cdots;A'_0)}\geq   \frac{1}{\sum_{\ell=1}^{m}\frac{L_{\ell}^*(A_{\ell}\cup \{i'_{n_0}\})}{L_{\ell}^*(A_{\ell})}}>1,   
\end{align*}
which proves (6.1) and \[ {\bf D}^*(A_1,\cdots, A_{m};A_0)>0. \]
\qed
\\
\\
{\textit{Proof of Theorem 3.3.}}\\
a. 
Let $A_0=\{i_0,i_1,\cdots,i_n\}.$ Note that
\begin{align*}&P( X_{i_1'}=\ell| \cap_{k=0,\cdots,m} \cap_{i\in A_k} \{ X_i=k\} )=\frac{{\bf D}^*(\cdots, A_{\ell}\cup \{i_1'\}, \cdots;A_0 )}{{\bf D}^*(A_1,\cdots,A_m;A_0 )}=\\&\frac{L_{\ell}^*(A_{\ell}\cup \{i_1'\}){\bf D}^*_{(-\ell)}(\cdots, A_{\ell-1},A_{\ell+1},\cdots;A_0)+\sum_{j=0}^{n}g_{\bf{H}, \bf{p},\bf c}(A_1,\cdots,A_{\ell}\cup\{i_1'\}, \cdots, A_m;A_0;i_j)}{L_{\ell}^*(A_{\ell}){\bf D}^*_{(-\ell)}(\cdots, A_{\ell-1},A_{\ell+1},\cdots;A_0)+\sum_{j=0}^{n}g_{\bf{H}, \bf{p},\bf c}(A_1,\cdots,A_{\ell}, \cdots, A_m;A_0;i_j)}.  \end{align*}
Since \[ \frac{g_{\bf{H}, \bf{p},\bf c}(A_1,\cdots,A_{\ell}\cup\{i_1'\}, \cdots, A_m;A_0;i_j)}{g_{\bf{H}, \bf{p},\bf c}(A_1,\cdots,A_{\ell}, \cdots, A_m;A_0;i_j)} \]
is non-decreasing as $j$ increases, and by (6.7-6.9)
\[ \frac{L_{\ell}^*(A_{\ell}\cup\{i_1'\})}{L_{\ell}^*(A_{\ell})}\leq \frac{g_{\bf{H}, \bf{p},\bf c}(A_1,\cdots,A_{\ell}\cup\{i_1'\}, \cdots, A_m;A_0;i_j)}{g_{\bf{H}, \bf{p},\bf c}(A_1,\cdots,A_{\ell}, \cdots, A_m;A_0;i_j)}, \]
 the result follows by Lemma 6.1 {\it ii}.
\\b.
\begin{align*}
& \frac{ P( X_{i_2'}=\ell| \cap_{k=0,\cdots,m} \cap_{i\in A_k} \{ X_i=k\})
}{    P( X_{i_3'}=\ell| \cap_{k=0,\cdots,m} \cap_{i\in A_k} \{ X_i=k\} )}=\frac{{\bf D}^*(\cdots, A_{\ell}\cup \{i_2'\}, \cdots;A_0 )}{{\bf D}^*(\cdots, A_{\ell}\cup \{i_3'\}, \cdots;A_0 )}=\\&\frac{L_{\ell}^*(A_{\ell}\cup \{i_2'\}){\bf D}^*_{(-\ell)}(\cdots, A_{\ell-1},A_{\ell+1},\cdots;A_0)+\sum_{j=0}^{n}g_{\bf{H}, \bf{p} ,\bf{c}^{\ell}}(A_1,\cdots,A_{\ell}\cup\{i_2'\}, \cdots, A_m;A_0;i_j)}{L_{\ell}^*(A_{\ell}\cup \{i_3'\}){\bf D}^*_{(-\ell)}(\cdots, A_{\ell-1},A_{\ell+1},\cdots;A_0)+\sum_{j=0}^{n}g_{\bf{H}, \bf{p}}(A_1,\cdots,A_{\ell}\cup\{i_3'\}, \cdots, A_m;A_0;i_j)}.
\end{align*}

For fixed $j, C$ such that $\max C <i_j$,
\begin{align*}
\frac{L_{\ell}^*(   A_{\ell}\cup\{i_2',i_j\} \cup C  )}{L_{\ell}^*(   A_{\ell}\cup \{i_3', i_j\}\cup C  )}
\leq 
\frac{L_{\ell}^*(   A_{\ell}\cup\{i_2'\}  )}{L_{\ell}^*(   A_{\ell}\cup\{i_3'\} )} ,
\end{align*}
and \[\frac{L_{\ell}^*(   A_{\ell}\cup\{i_2',i_j\} \cup C  )}{L_{\ell}^*(   A_{\ell}\cup \{i_3', i_j\}\cup C  )}\] is non-increasing as $j$ increases. 
Therefore, the result follows by Lemma 6.1 {\it i}.
\qed

\end{document}